\documentclass[11 pt]{article}
\usepackage{latexsym}
\usepackage{graphicx}
\usepackage{amsmath}
\usepackage{amsthm}
\usepackage{amssymb,amsfonts,enumerate}
\usepackage[pdfpagelabels=true,plainpages=false,colorlinks=true,linkcolor=blue,citecolor=blue,urlcolor=blue]{hyperref}
\newtheorem{thm}{Theorem}[section]

\newtheorem{df}{Definition}[section]
\newtheorem{remark}{Remark}[section]
\newtheorem{exmp}{Example}[section]

\begin{document}
\begin{center}
{\Large{\bf Some characterizations on weighted $\alpha\beta$-statistical convergence of fuzzy functions of order $\theta$}}\\\vspace{.5cm}
Sarita Ojha and P. D. Srivastava\\
Department of Mathematics, Indian Institute of Technology,\\ Kharagpur 721302, India
\end{center}
\vspace{1cm}

\section*{Abstract}
Based on the concept of new type of statistical convergence defined by Aktuglu, we have introduced the weighted $\alpha\beta$ - statistical convergence of order $\theta$  in case of fuzzy functions and classified it into pointwise, uniform and equi-statistical convergence. We have checked some basic properties and then the convergence are investigated in terms of their $\alpha$-cuts. The interrelation among them are also established. We have also proved that continuity, boundedness etc are preserved in the equi-statistical sense under some suitable conditions, but not in pointwise sense.\\
{\bf Keywords:} Sequences of fuzzy numbers; Fuzzy function; Weighted statistical convergence.\\\vspace{.4cm}
{\bf AMS subject classification:} 46S40;  03E72.

\section{Introduction}
Convergence of sequences, in classical or fuzzy sense, means that almost all elements of the sequence have to belong to an arbitrary small neighborhood of the limit. The aim is to introduce the concept of statistical convergence is to relax the above condition and to examine the convergence criterion only for majority of elements. As always, statistics are concerned only about big quantities and the term "majority" is simply imply the concept "almost all" in our known classical analysis.\\
On the other hand, as the set of all real numbers can be embedded in the set of all fuzzy numbers, statistical convergence in reals can be considered as a special case of those fuzzy numbers. But the set of all fuzzy numbers is partially ordered and does not carry a group structure, so most of the results known for the sequences of real numbers may not be valid in fuzzy setting. Therefore, the theory is not a trivial extension of what has been known in real case.\\
The concept of statistical convergence was introduced by Fast \cite{F} in 1951. Several development in this direction, for reals as well as for fuzzy, can be enumerated as follows:
\begin{enumerate}[(i)]
\item $\lambda$ - statistical convergence of order $\alpha$ \cite{CB}, \cite{At}.
\item lacunary statistical convergence order $\alpha$ \cite{N}, \cite{AG}.
\item weighted statistical convergence \cite{MM}.
\end{enumerate}
The main aim of introducing $\alpha\beta$ - statistical convergence was to extend the ordinary and statistical convergence sense. This not only includes some well known matrix methods such as statistical convergence, $\lambda$ - statistical convergence, lacunary statistical convergence but also gives some non-regular matrix methods.

\section{Preliminaries}
A fuzzy real number $\hat{x}:\mathbb{R}\rightarrow [0,1]$ is a fuzzy set which is normal, fuzzy convex, upper semicontinuous and $[\hat{x}]_0=\{t\in \mathbb{R}:\hat{x}(t)>0\}$ is compact. Clearly, $\mathbb{R}$ is embedded in $L(R)$, the set of all fuzzy numbers, in this way: for each $r\in \mathbb{R}, \overline{r}\in L(R)$ is defined as,
\begin{center}
$\overline{r}(t) = \left\{
\begin{array}{c l}
  1, & t=r \\
  0, & t\neq r
\end{array}
\right.$\end{center}
{\bf Note:} Throughout this article, we shall use this type of fuzzy numbers to avoid complications in calculations.\\
For, $0<\alpha\leq1$, $\alpha$-cut of $\hat{x}$ is defined by $[\hat{x}]_{(\alpha)}=\{t\in\mathbb{R}:\hat{x}(t)\geq\alpha\}=[(\hat{x})_{\alpha} ^-,(\hat{x})_{\alpha} ^+]$ is a closed and bounded interval of $\mathbb{R}$. Now for any two fuzzy numbers $\hat{x},\hat{y}$, Matloka \cite{M} proved that $L(R)$ is complete under the metric $d$ defined as:
\begin{equation*}
d(\hat{x},\hat{y})=\sup\limits_{0\leq \alpha \leq1} \max\{|(\hat{x})_{\alpha} ^--(\hat{y})_{\alpha} ^-|,|(\hat{x})_{\alpha} ^+-(\hat{y})_{\alpha} ^+|\}
\end{equation*}
For any $\hat{x},\hat{y},\hat{z},\hat{w}\in L(R)$, this Hausdorff metric $d$ satisfies
\begin{enumerate}[(i)]
\item $d(c\hat{x},c\hat{y})=|c| d(\hat{x},\hat{y})$, $c\in\mathbb{R}$.
\item $d(\hat{x}+\hat{z},\hat{y}+\hat{z})=d(\hat{x},\hat{y})$.
\item $d(\hat{x}+\hat{z},\hat{y}+\hat{w})\leq d(\hat{x},\hat{y})+d(\hat{z},\hat{w})$.
\end{enumerate}

\begin{df}
A sequence $x=(x_k)$ of real or complex numbers is said to be statistically convergent to a
number $l$ if for every $\varepsilon>0$,
\begin{equation*}
\lim\limits_{n\rightarrow\infty} \frac{1}{n}|\{k\leq n:|x_k-l|\geq
\varepsilon\}|=0.
\end{equation*}
\end{df}
\noindent where the vertical bars indicate the number of elements in the enclosed set. This has been extended in case of fuzzy by several authors such as Nuray \cite{N}, Savas \cite{NS}, Ojha and Srivastava \cite{SO} etc from several aspects. Recently, Gong et al \cite{ZZZ} defined statistical convergence in measure for sequences of fuzzy functions and established Egorov and Lebesgue theorems in statistical sense.

\begin{df}\cite{MM}
Let $(t_n)$ be a sequence of non-negative real numbers such that $t_1>0$ and $T_n=\sum\limits_{k=1} ^n t_k\to\infty$ as $n\to\infty$. Then the sequence of real numbers $(x_n)$ is said to be weighted statistically convergent to $l$ if for every $\varepsilon>0$,
\begin{equation*}
\lim\limits_{n\rightarrow\infty} \frac{1}{T_n}|\{k\leq T_n:t_k|x_k-l|\geq
\varepsilon\}|=0.
\end{equation*}
\end{df}

\noindent Many authors have generalized the definition of statistical convergence and introduced new class of sequences for the real or complex numbers. The most recent generalization in this direction is the idea of $\alpha\beta$-statistical convergence which is introduced by Aktuglu as follows:
\begin{df} \cite{A}
Let $(\alpha_n),(\beta_n)$ be two sequence of positive number s.t.
\begin{enumerate}[(i)]
\item $(\alpha_n)$, $(\beta_n)$ are both non-decreasing.
\item $\beta_n\geq \alpha_n$.
\item $\beta_n-\alpha_n\to\infty$ as $n\to\infty$.
\end{enumerate}
Then a sequence $(x_k)$ of real numbers is said to be $\alpha\beta$ - statistically convergent of order $\gamma$ to $L$ if for every $\varepsilon>0$,
\begin{equation*}
\lim\limits_{n\to\infty} \frac{1}{(\beta_n-\alpha_n+1)^{\gamma}} \Big|\Big\{k\in[\alpha_n,\beta_n]: |x_k-L|\geq\varepsilon\Big\}\Big|\to 0.
\end{equation*}
\end{df}

\begin{remark}
This definition includes the following cases:
\begin{enumerate}[(i)]
\item Taking $\alpha_n=1,\beta_n=n$, this $\alpha\beta$-convergence coincides with the usual statistical convergence \cite{F}.
\item Let $(\lambda_n)$ is non-decreasing sequence of positive real numbers tending to $\infty$ such that $\lambda_1=1,\lambda_{n+1}\leq \lambda_n+1$ for all $n$. Then choosing $\alpha_n=n-\lambda_n+1$ and $\beta_n=n$, this convergence coincides with the concept of $\lambda$-statistical convergence of order $\theta$ \cite{CB}.
\item Choosing $\alpha_r=k_{r-1}+1,\beta_r=k_r$ where $\theta=(k_r)$ is an increasing integer sequence with $k_0=0$ and $h_r=k_r-k_{r-1}$, then this convergence coincides with lacunary statistical convergence of order $\gamma$ \cite{AG}.
\end{enumerate}
\end{remark}

\noindent Then Karakaya \cite{KK} defined weighted $\alpha\beta$ -statistical convergence and summability for real sequences. Recently Ghoshal \cite{G} has introduced weighted random convergence in probability. He has shown that the definition of weighted $\lambda$ - statistical convergence, given by \cite{BM} is not well defined, so he modified the condition on the weight sequences $(t_n)$. With this modified condition, we define the three types of convergence of fuzzy functions in weighted $\alpha\beta$-statistical sense and investigate various properties of it.

\section{$\alpha\beta$-statistical convergence of sequence of fuzzy functions of order $\theta$}

Let $t=(t_k)$ be a sequence of non-negative real numbers such that $\lim\inf _k t_k>0$ and
\begin{equation*}
T_{\alpha\beta (n)}=\sum\limits_{k\in  [\alpha_n,\beta_n]} t_k,\ n\in\mathbb{N}
\end{equation*}

\begin{df}
A sequence $\hat{f_k}:[a,b]\to L(R)$ of fuzzy functions is said to be weighted $\alpha\beta$ - pointwise statistical convergent of order $\theta$ to $\hat{f}$ if for every $\varepsilon>0$ and for each $x\in[a,b]$,
\begin{equation*}
\lim\limits_{n\to\infty} \frac{1}{T_{\alpha\beta (n)} ^{\theta}} \Big|\Big\{k\leq T_{\alpha\beta (n)}: t_kd(\hat{f_k}(x),\hat{f}(x))\geq\varepsilon\Big\}\Big|=0
\end{equation*}
\end{df}
\noindent We shall denote it as $\hat{f_k}\xrightarrow{WP_{\alpha,\beta} ^{\theta}} \hat{f}$ on $[a,b]$.

\begin{df}
A sequence $\hat{f_k}:[a,b]\to L(R)$ of fuzzy functions is said to be weighted $\alpha\beta$ - uniformly statistical convergent of order $\theta$ to $\hat{f}$ if for every $\varepsilon>0$ and for all $x\in[a,b]$,
\begin{equation*}
\lim\limits_{n\to\infty} \frac{1}{T_{\alpha\beta (n)} ^{\theta}} \Big|\Big\{k\leq T_{\alpha\beta (n)}: t_kd(\hat{f_k}(x),\hat{f}(x))\geq\varepsilon\Big\}\Big|=0
\end{equation*}
\end{df}
\noindent We shall denote it as $\hat{f_k}\xrightarrow{WU_{\alpha,\beta} ^{\theta}} \hat{f}$ on $[a,b]$.

\begin{df}
A sequence $\hat{f_k}:[a,b]\to L(R)$ of fuzzy functions is said to be weighted $\alpha\beta$ - equi-statistically convergent of order $\theta$ to $\hat{f}$ if for given $\varepsilon>0$,
\begin{equation*}
S_m (x)= \frac{1}{T_{\alpha\beta (m)} ^{\theta}} \Big|\Big\{k\leq T_{\alpha\beta (m)}: t_kd(\hat{f_k}(x),\hat{f}(x))\geq\varepsilon\Big\}\Big|
\end{equation*}
is uniformly converges to 0 with respect to $x\in [a,b]$. We shall denote it as $\hat{f_k}\xrightarrow{WE_{\alpha,\beta} ^{\theta}} \hat{f}$.
\end{df}

\begin{thm}
Let $(\hat{f_k}), (\hat{g_k})$ be two sequences of functions from $[a,b]$ to $L(R)$ such that $\hat{f_k}\xrightarrow{WP_{\alpha,\beta} ^{\theta}} \hat{f}$ and $\hat{g_k}\xrightarrow{WP_{\alpha,\beta} ^{\theta}} \hat{g}$ on $[a,b]$. Then,
\begin{enumerate}[(i)]
\item $\hat{f_k}+\hat{g_k} \xrightarrow{WP_{\alpha,\beta} ^{\theta}} \hat{f}+\hat{g}$.
\item $c\hat{f_k} \xrightarrow{WP_{\alpha,\beta} ^{\theta}} c\hat{f}$ for any non-zero $c\in\mathbb{R}$.
\end{enumerate}
\end{thm}
\begin{proof}
\begin{enumerate}[(i)]
\item Since
\begin{eqnarray*}
d(\hat{f_k}(x)+\hat{g_k}(x),\hat{f}(x)+\hat{g}(x)) \leq d(\hat{f_k}(x),\hat{f}(x))+d(\hat{g_k}(x),\hat{g}(x))\ \mbox{for each}\  x\in[a,b]
\end{eqnarray*}
Also since $(t_k)$ are positive real numbers, so
\begin{eqnarray*}
t_k[d(\hat{f_k}(x),\hat{f}(x))+d(\hat{g_k}(x),\hat{g}(x))]\geq \varepsilon \implies t_kd(\hat{f_k}(x),\hat{f}(x))\geq\frac{\varepsilon}{2}\ \mbox{or} \ t_kd(\hat{g_k}(x),\hat{g}(x))\geq\frac{\varepsilon}{2}
\end{eqnarray*}
So for each $x\in[a,b]$,
\begin{eqnarray*}
\Big\{k\leq T_{\alpha\beta (n)}: t_kd(\hat{f_k}(x)+\hat{g_k}(x),\hat{f}(x)+\hat{g}(x))\geq\varepsilon\Big\} &\subseteq&  \Big\{k\leq T_{\alpha\beta (n)}: t_kd(\hat{f_k}(x),\hat{f}(x))\geq\frac{\varepsilon}{2}\Big\}\\ && \cup \ \Big\{k\leq T_{\alpha\beta (n)}: t_kd(\hat{g_k}(x),\hat{g}(x))\geq\frac{\varepsilon}{2}\Big\}
\end{eqnarray*}
Taking cardinality in both sides, we get $\hat{f_k}+\hat{g_k} \xrightarrow{WP_{\alpha,\beta} ^{\theta}} \hat{f}+\hat{g}$.
\item As for any non-zero $c\in\mathbb{R}$ and $x\in[a,b]$, $d(c\hat{f_k}(x),c\hat{f}(x))=|c|d(\hat{f_k}(x),\hat{f}(x))$, thus for each $x\in[a,b]$,
\begin{eqnarray*}
\Big\{k\leq T_{\alpha\beta (n)}: t_kd(c\hat{f_k}(x),c\hat{f}(x))\geq\varepsilon\Big\}  =  \Big\{k\leq T_{\alpha\beta (n)}: t_kd(\hat{f_k}(x),\hat{f}(x))\geq\frac{\varepsilon}{|c|}\Big\}
\end{eqnarray*}
So $c\hat{f_k} \xrightarrow{WP_{\alpha,\beta} ^{\theta}} c\hat{f}$ which completes the proof.
\end{enumerate}
\end{proof}

\begin{thm}
Let $(\hat{f_k})$ be a sequences of functions from $[a,b]$ to $L(R)$ such that $\hat{f_k}\xrightarrow{WP_{\alpha,\beta} ^{\theta}} \hat{f}$ on $[a,b]$ and $[c,d]\subset [a,b]$, then $\hat{f_k}\xrightarrow{WP_{\alpha,\beta} ^{\theta}} \hat{f}$ on $[c,d]$.
\end{thm}

\begin{thm}
Let $0<\theta\leq \gamma\leq1$. Then on $[a,b]$
\begin{equation*}
\hat{f_k}\xrightarrow{WP_{\alpha,\beta} ^{\theta}} \hat{f} \implies \hat{f_k}\xrightarrow{WP_{\alpha,\beta} ^{\gamma}} \hat{f}.
\end{equation*}
The inclusion is strict in sense. This holds for Definition 2.2 and 2.3 also.
\end{thm}
\begin{proof}
The proof is simple, so we omit it. To prove the converse, let us consider a counter example.
\end{proof}
\begin{exmp}
Consider the interval $[\alpha_n,\beta_n]$. Define for $x\in[a,b]$
\begin{center}
$f_k(x) = \left\{
\begin{array}{c l}
  \bar{1}, & k=n^2 \ \mbox{for some} \ n\in\mathbb{N}\\
  \bar{0}, & \mbox{otherwise}.
\end{array}
\right.$
\end{center}
Then taking $t_k=1$ for all $k$, we get for each $x\in[a,b]$
\begin{eqnarray*}
&& \frac{1}{(\beta_n-\alpha_n+1)^{\gamma}}\Big|\Big\{k\leq (\beta_n-\alpha_n+1): d(\hat{f_k}(x),\bar{0})\geq\varepsilon\Big\}\Big|\\
&& = \frac{[(\beta_n-\alpha_n+1)^{1/2}]}{(\beta_n-\alpha_n+1)^{\gamma}}
\end{eqnarray*}
Now \begin{eqnarray*}
\frac{(\beta_n-\alpha_n+1)^{1/2}-1}{(\beta_n-\alpha_n+1)^{\gamma}}\leq\frac{[(\beta_n-\alpha_n+1)^{1/2}]}{(\beta_n-\alpha_n+1)^{\gamma}}\leq \frac{1}{(\beta_n-\alpha_n+1)^{\gamma-1/2}}
\end{eqnarray*}
which tends to 0 as $n\to \infty$ when $\gamma>1/2$ and diverges for $\gamma<1/2$. So by choosing $\theta<1/2<\gamma$, it is clear that $\hat{f_k}\xrightarrow{WP_{\alpha,\beta} ^{\gamma}} \hat{f}$ but $\hat{f_k}$ is not weighted $\alpha\beta$ - pointwise statistical convergent to $\bar{0}$ of order $\theta$.
\end{exmp}

\begin{remark}
The above three Theorems (2.1-2.3) is true for Definition 2.2 and Definition 2.3 also.
\end{remark}

\begin{thm}
If $\hat{f_k}\xrightarrow{WU_{\alpha,\beta} ^{\theta}} \hat{f}$ on $[a,b]$ $\iff$ $\sup\limits_{x\in[a,b]} t_kd(\hat{f_k}(x),\hat{f}(x))\xrightarrow{WP_{\alpha,\beta} ^{\theta}} 0$.
\end{thm}
\begin{proof}
Let $\sup\limits_{x\in[a,b]} t_kd(\hat{f_k}(x),\hat{f}(x))\xrightarrow{WP_{\alpha,\beta} ^{\theta}} 0$. Now for every $\varepsilon>0$,
\begin{eqnarray*}
&& \sup\limits_{x\in[a,b]} t_kd(\hat{f_k}(x),\hat{f}(x))\geq t_kd(\hat{f_k}(x),\hat{f}(x))\ \mbox{for all}\ x\in[a,b]\\
\mbox{i.e.}&& \ \Big\{k\leq T_{\alpha\beta (n)}: t_kd(\hat{f_k}(x),\hat{f}(x))\geq\varepsilon \ \mbox{for all}\ x\in[a,b]\Big\}\subseteq \Big\{k\leq T_{\alpha\beta (n)}: \sup\limits_{x\in[a,b]} t_kd(\hat{f_k}(x),\hat{f}(x))\geq\varepsilon\Big\}
\end{eqnarray*}
So, $\hat{f_k}\xrightarrow{WU_{\alpha,\beta} ^{\theta}} \hat{f}$ on $[a,b]$.\\
Conversely, let $\hat{f_k}\xrightarrow{WU_{\alpha,\beta} ^{\theta}} \hat{f}$ on $[a,b]$. For any $\varepsilon>0$, consider the sets,\\
\begin{eqnarray*}
A_n &=& \Big\{k\leq T_{\alpha\beta (n)}: \sup\limits_{x\in[a,b]} t_kd(\hat{f_k}(x),\hat{f}(x))\geq\varepsilon\Big\}\\
B_n &=& \Big\{k\leq T_{\alpha\beta (n)}: t_kd(\hat{f_k}(x),\hat{f}(x))\geq\varepsilon \ \mbox{for all}\ x\in[a,b]\Big\}
\end{eqnarray*}
\begin{eqnarray*}
\mbox{Let}\ k\in B_n ^c &\implies& t_kd(\hat{f_k}(x),\hat{f}(x))<\varepsilon\ \forall \ x\in[a,b]\\
&\implies& \sup\limits_{x\in[a,b]} t_kd(\hat{f_k}(x),\hat{f}(x))<\varepsilon\ \mbox{as $\varepsilon$ is arbitrary}\\
&\implies& k\in A_n ^c\\
B_n ^c \subseteq A_n ^c &\implies& A_n\subseteq B_n\\
&\implies& \frac{1}{T ^{\theta} _{\alpha\beta (n) }} |A_n|\leq \frac{1}{T ^{\theta} _{\alpha\beta (n) }} |B_n|\to 0\ \mbox{as}\ n\to\infty
\end{eqnarray*}
$\therefore \ \sup\limits_{x\in[a,b]} t_kd(\hat{f_k}(x),\hat{f}(x))\xrightarrow{WP_{\alpha,\beta} ^{\theta}} 0$.
\end{proof}

\begin{thm}
It is obvious that $\hat{f_k}\xrightarrow{WU_{\alpha,\beta} ^{\theta}} \hat{f}\ \Rightarrow \ \hat{f_k}\xrightarrow{WE_{\alpha,\beta} ^{\theta}} \hat{f} \ \Rightarrow \ \hat{f_k}\xrightarrow{WP_{\alpha,\beta} ^{\theta}} \hat{f}$ on $[a,b]$ as $k\to\infty$, for each $0<\theta\leq 1$. The inverse implications does not hold in general.
\end{thm}
\begin{proof}
To show this, we shall give the following examples. Let $t_n=1$ for all $n$ and $\theta=1$.
\begin{exmp}
Let the sequence of fuzzy functions is defined by $\hat{f} _n(x)=\overline{(e^{-nx})}$, $x\in [0,1]$. Then $d(\hat{f} _n(x),\bar{0})=e^{-nx}$ for each $n\in\mathbb{N}$. Consider $\alpha_n=n,\beta_n=2n-1$. Then $T_{\alpha\beta(n)}=n$.\\
So, $\hat{f}_n \to \bar{0}$ and therefore $\hat{f}_n \xrightarrow{st} \bar{0}$.\\
Now for each $m\in\mathbb{N}$, consider the interval $k\in[m,2m-1]$. Then for all $x\in [0,\frac{1}{2m-1}]$,
\begin{eqnarray*}
d(\hat{f} _k(x),\bar{0}) &=& e^{-kx}\geq e^{-(2m-1)x} \ \mbox{(since $e^{-kx}$ is monotonically decreasing for $k$)}\\
&\geq& \frac{1}{e}\ \mbox{as $x\in[0,\frac{1}{2m-1}]$}\\
&\geq& \frac{1}{3}.
\end{eqnarray*}
So for all $x\in [0,\frac{1}{2m-1}]$,
\begin{eqnarray*}
S_m (x)=\frac{1}{m}\Big|\Big\{k\in[m,2m-1]: d(\hat{f}_k(x),\bar{0})\geq \frac{1}{3}\Big\}\Big|\to1(\neq 0)
\end{eqnarray*}
So, $(\hat{f_k})$ does not weighted equi-statistically $\alpha\beta$-statistical convergent to $\bar{0}$.
\end{exmp}
\begin{exmp}
Let $\alpha_n=1, \beta_n=n$ for each $n$. For any $x\in [0,1]$, define\\
$f_n(x) = \left\{
\begin{array}{c l}
  \overline{\Big(\frac{nx}{1+n^2x^2}\Big)}, & x\in[\frac{1}{n+2},\frac{1}{{n+1}}] \\
  \bar{0}, & \mbox{otherwise}.
\end{array}
\right.$\\
Then for every $x\in[0,1]$, $|\{n\in \mathbb{N};\hat{f}_n(x)\neq\bar{0}\}|=0$. Thus for any given $\varepsilon>0$,
\begin{eqnarray*}
\frac{1}{m} |\{k\leq m: d(\hat{f_k}(x),\bar{0})\geq\varepsilon\}|\leq \frac{|\{n\in \mathbb{N};\hat{f}_n(x)\neq\bar{0}\}\cap\{1,2,\ldots,m\}|}{m}\leq \frac{1}{m}\to 0
\end{eqnarray*}
Hence $\hat{f_k}\xrightarrow{eq} \bar{0}$. But for all $n$,
\begin{eqnarray*}
\sup\limits_{x\in[0,1]} d(\hat{f}_n(x),\bar{0})=\sup\limits_{x\in[0,1]} \frac{nx}{1+n^2x^2} =\sup\limits_{x\in[0,1]} \frac{1}{\frac{1}{nx}+nx}=\frac{1}{2}.
\end{eqnarray*}
So $(\hat{f}_k)$ is not uniformly statistical convergent to $\bar{0}$.
\end{exmp}
\end{proof}

\begin{thm}
Let $(\hat{f_k})$ be a sequence of fuzzy functions. Then the followings hold:
\begin{enumerate}[(i)]
\item $\hat{f_k}\xrightarrow{WP_{\alpha,\beta} ^{\theta}} \hat{f}$ on $[a,b]$ iff $[\hat{f_k}]_{\alpha}$ is weighted $\alpha\beta$ - uniformly statistically convergent to $[\hat{f}]_{\alpha}$ w.r.t. $\alpha$.
\item $\hat{f_k}\xrightarrow{WU_{\alpha,\beta} ^{\theta}} \hat{f}$ on $[a,b]$ iff $[\hat{f_k}(x)]_{\alpha}$ is weighted $\alpha\beta$ - uniformly statistically convergent to $[\hat{f}(x)]_{\alpha}$ w.r.t. $\alpha$ and $x$.
\item $\hat{f_k}\xrightarrow{WE_{\alpha,\beta} ^{\theta}} \hat{f}$ on $[a,b]$ iff $[\hat{f_k}(x)]_{\alpha}$ is weighted $\alpha\beta$ - uniformly equi-statistically convergent to $[\hat{f}(x)]_{\alpha}$ for any $\alpha$ and for any $x\in[a,b]$.
\end{enumerate}
\end{thm}
\begin{proof}
We shall prove this only for (ii). As (i),(iii) follow almost the same lines.\\
(ii) Let $\hat{f_k}\xrightarrow{WU_{\alpha,\beta} ^{\theta}} \hat{f}$. Now for all $\alpha\in[0,1]$ and for all $x\in[a,b]$,
\begin{eqnarray*}
t_k|(\hat{f_k}(x))_{\alpha} ^+ - (\hat{f}(x))_{\alpha} ^+|\leq t_kd(\hat{f_k}(x),\hat{f}(x))\ \mbox{and} \ t_k|(\hat{f_k}(x))_{\alpha} ^- - (\hat{f}(x))_{\alpha} ^-|\leq t_kd(\hat{f_k}(x),\hat{f}(x))
\end{eqnarray*}
So for all $\varepsilon>0$,
\begin{eqnarray*}
&& \Big\{k\leq T_{\alpha\beta(n)}:t_k|(\hat{f_k}(x))_{\alpha} ^+ - (\hat{f}(x))_{\alpha} ^+|\geq \varepsilon\ \forall \ x\in[a,b]\Big\}\\ && \subseteq \Big\{k\leq T_{\alpha\beta(n)}:t_kd(\hat{f_k}(x),\hat{f}(x))\geq \varepsilon\ \forall \ x\in[a,b]\Big\}\\
&\mbox{and}& \Big\{k\leq T_{\alpha\beta(n)}:t_k|(\hat{f_k}(x))_{\alpha} ^- - (\hat{f}(x))_{\alpha} ^-|\geq \varepsilon\ \forall \ x\in[a,b]\Big\}\\ && \subseteq \Big\{k\leq T_{\alpha\beta(n)}:t_kd(\hat{f_k}(x),\hat{f}(x))\geq \varepsilon\ \forall \ x\in[a,b]\Big\}
\end{eqnarray*}
Therefore $[\hat{f_k}(x)]_{\alpha}$ is weighted $\alpha\beta$ - uniformly statistically convergent to $[\hat{f}(x)]_{\alpha}$ w.r.t. $\alpha$ and $x$.\\
Conversely, let $(\hat{f_k}(x))_{\alpha}$ is weighted $\alpha\beta$ - uniformly convergent to $\hat{f}$ w.r.t. $\alpha$ and $x$. Then for any $\varepsilon>0$ and for all $x\in[a,b]$, consider the sets\\
\begin{eqnarray*}
A_n &=& \Big\{k\leq T_{\alpha\beta(n)}:t_k|(\hat{f_k}(x))_{\alpha} ^+ - (\hat{f}(x))_{\alpha} ^+|\geq \varepsilon\Big\}\\
B_n &=& \Big\{k\leq T_{\alpha\beta(n)}:t_k|(\hat{f_k}(x))_{\alpha} ^- - (\hat{f}(x))_{\alpha} ^-|\geq \varepsilon\Big\}\\
C_n &=& \Big\{k\leq T_{\alpha\beta(n)}:t_kd(\hat{f_k}(x),\hat{f}(x))\geq \varepsilon\Big\}.
\end{eqnarray*}
For any $k\in A_n ^c\cap B_n^c$
\begin{eqnarray*}
t_k|(\hat{f_k}(x))_{\alpha} ^+ - (\hat{f}(x))_{\alpha} ^+|< \varepsilon\ \mbox{and} \ t_k|(\hat{f_k}(x))_{\alpha} ^- - (\hat{f}(x))_{\alpha} ^-|< \varepsilon
\end{eqnarray*}
for any $x\in[a,b]$ and $\alpha\in[0,1]$. Thus
\begin{eqnarray*}
&& t_kd(\hat{f_k}(x),\hat{f}(x))=t_k\sup\limits_{\alpha\in[0,1]}\max\{|(\hat{f_k}(x))_{\alpha} ^+ - (\hat{f}(x))_{\alpha} ^+|,|(\hat{f_k}(x))_{\alpha} ^- - (\hat{f}(x))_{\alpha} ^-|\}<\varepsilon\\
\mbox{i.e.} && \ k\in C_n ^c.
\end{eqnarray*}
Applying set theoretic approach, we have $|C_n|\leq |A_n|+|B_n|$.\\ So, dividing by $T_{\alpha\beta(n)} ^{\theta}$ in both sides, we get $\hat{f_k}\xrightarrow{WU_{\alpha,\beta} ^{\theta}} \hat{f}$. This completes the proof.
\end{proof}

\begin{thm}
Let $(\hat{f_k})$ be a sequence of fuzzy functions  from $[a,b]$ to $L(R)$ such that $\hat{f_k}\xrightarrow{WE_{\alpha,\beta} ^{\theta}} \hat{f}$ on $[a,b]$. Then for any bounded weight sequence $(t_n)$, the following hold:
\begin{enumerate}[(i)]
\item $\lim\limits_{x\to c} \hat{f_k} \xrightarrow{WE_{\alpha,\beta} ^{\theta}} \lim\limits_{x\to c} \hat{f}$.
\item If $(\hat{f_k})$ are continuous at some point $c\in[a,b]$, then $\hat{f}$ is continuous at $c$.
\item If $(\hat{f_k})$ be uniformly bounded on $[a,b]$, then $\hat{f}$ is also bounded on $[a,b]$.
\end{enumerate}
\end{thm}
\begin{proof}
Let $t_n<M$ for all $n$. Since $\hat{f_k}\xrightarrow{WE_{\alpha,\beta} ^{\theta}} \hat{f}$ on $[a,b]$, then for any given $\varepsilon>0$ there exists $p\in \mathbb{N}$ such that
\begin{equation*}
\frac{1}{T_{\alpha\beta (p)} ^{\theta}} \Big|\Big\{k\leq T_{\alpha\beta (p)}: t_kd(\hat{f_k}(x),\hat{f}(x))\geq \frac{\varepsilon}{3M}\Big\}\Big|<\frac{1}{2}
\end{equation*}
for all $x\in [a,b]$. Consider the set $A_p(x)=\Big\{k\leq T_{\alpha\beta (p)}: t_kd(\hat{f_k}(x),\hat{f}(x))< \frac{\varepsilon}{3M}\Big\}$, $x\in[a,b]$. Then $\frac{1}{T_{\alpha\beta (p)} ^{\theta}}|A_p(x)|>\frac{1}{2}$ for all $x\in[a,b]$. \\
Let $\lim\limits_{x\to c} \hat{f_k}(x)=\hat{u_k}$. Then for some $\delta_1>0$,
\begin{equation*}
c-\delta_1<x<c+\delta_1 \implies d(\hat{f_k}(x),\hat{u_k})< \frac{\varepsilon}{3M}
\end{equation*}
Denote $\lim\limits_{x\to c} \hat{f}(x)=\hat{u}$. Then for some $\delta_2>0$,
\begin{equation*}
c-\delta_2<x<c+\delta_2 \implies d(\hat{f}(x),\hat{u})< \frac{\varepsilon}{3M}
\end{equation*}
Set $\delta=\max\{\delta_1,\delta_2\}$ and take $n\in A_p (x)$. Then $\delta>0$ and for all  $x$ satisfying $c-\delta<x<c+\delta$,
\begin{eqnarray*}
d(\hat{u_n},\hat{u}) &\leq& d(\hat{u_n},\hat{f_n}(x))+d(\hat{f_n}(x),\hat{f}(x))+d(\hat{f}(x),\hat{u})\\
\mbox{Therefore} \ t_nd(\hat{u_n},\hat{u}) &<& M[\frac{\varepsilon}{3M}+\frac{\varepsilon}{3M}+\frac{\varepsilon}{3M}]=\varepsilon.
\end{eqnarray*}
So, $n\in A_p(x)\implies t_nd(\hat{u_n},\hat{u})<\varepsilon$. Or in other words,
\begin{equation*}
\Big\{k\leq T_{\alpha\beta (p)}: t_kd(\hat{u_k},\hat{u})\geq \varepsilon\Big\}\subset \Big\{k\leq T_{\alpha\beta (p)}: t_kd(\hat{f_k}(x),\hat{f}(x))< \frac{\varepsilon}{3}\Big\}
\end{equation*}
Since $\varepsilon$ is arbitrary, so $\hat{u_k} \xrightarrow{WE_{\alpha,\beta} ^{\theta}} \hat{u}$. This completes the proof.\\
\end{proof}

\begin{remark}
Since the above result is true in case of equi-statistical convergence and equi-statistical convergence implies uniform convergence, so it also holds for weighted $\alpha\beta$ - uniform convergence of order $\theta$. But this does not hold for pointwise convergence. It is enough to give an example for any one of the three.
\end{remark}
\begin{exmp}
Let $(\hat{f_n})$ be a sequence of fuzzy functions defined as
\begin{equation*}
\hat{f_n}(x)=\overline{(x^n)},x\in[0,1]
\end{equation*}
Then for each $\alpha\in[0,1]$, $[\hat{f_n}(x)]_{\alpha}=x^n$. It is convergent and hence statistically convergent i.e. $\hat{f_n} \xrightarrow{st} \hat{f}(x)$ as $n\to\infty$ where $f(x) = \left\{
\begin{array}{c l}
  \bar{0}, & x\in[0,1) \\
  \bar{1}, & x=1.
\end{array}
\right.$.\\
Clearly, $(\hat{f_n}(x))$ are continuous function in $[0,1]$ but the limit function $\hat{f}(x)$ is not.\\
Rest of the cases can also be verified similarly.
\end{exmp}

\begin{thm}
Let $(\hat{f_k})$ be a sequence of equi-continuous fuzzy functions on $[a,b]$ such that $\hat{f_k}\xrightarrow{WP_{\alpha,\beta} ^{\theta}} \hat{f}$. Then for a bounded weight $(t_n)$, the limit function $\hat{f}$ is continuous on $[a,b]$ and $\hat{f_k}\xrightarrow{WU_{\alpha,\beta} ^{\theta}} \hat{f}$ on $[a,b]$.
\end{thm}
\begin{proof}
Let $L<t_n<M$ for all $n$. Clearly, $L>0$ according to our construction of $(t_n)$ i.e. $\lim\inf_{n\to\infty} t_n>0$. To prove $\hat{f}$ is continuous on $[a,b]$, let $c$ be any point in $[a,b]$. Since $(\hat{f_k})$ are equi- continuous on $[a,b]$, so $\exists$ $\delta>0$ such that
\begin{equation*}
d(\hat{f_n}(x),\hat{f_n}(c))<\frac{\varepsilon}{3}\ \mbox{whenever}\ c-\delta<x<c+\delta \ \forall \ n
\end{equation*}
Also since $\hat{f_k}\xrightarrow{WP_{\alpha,\beta} ^{\theta}} \hat{f}$, so there exists $p\in \mathbb{N}$ such that $t_pd(\hat{f_p}(x),\hat{f}(x))<\frac{L\varepsilon}{3}$ so that $d(\hat{f_p}(x),\hat{f}(x))<\frac{\varepsilon}{3}$.
Then,
\begin{equation*}
d(\hat{f}(x),\hat{f}(c))\leq d(\hat{f}(x),\hat{f_p}(x))+d(\hat{f_p}(x),\hat{f_p}(c))+d(\hat{f_p}(c),\hat{f}(c))<\frac{\varepsilon}{3}+\frac{\varepsilon}{3}+\frac{\varepsilon}{3}=\varepsilon
\end{equation*}
Since $c$ is arbitrary, so $\hat{f}$ is continuous on $[a,b]$.\\
Next we shall prove that the convergence will be uniform. Now, continuity in $[a,b]$ implies uniform continuity and same for equi-continuity also i.e. for any two points $x,y\in[a,b]$ and $\varepsilon>0$ there exists $\delta>0$ such that $|x-y|<\delta$ implies
\begin{center}
$d(\hat{f_k}(x),\hat{f_k}(y))< \frac{\varepsilon}{3M}$ and $d(\hat{f}(x),\hat{f}(y))< \frac{\varepsilon}{3M}$.
\end{center}
Since $[a,b]$ is compact, we can choose a finite cover $(x_1-\delta,x_1+\delta),(x_2-\delta,x_2+\delta),\ldots,(x_m-\delta,x_m+\delta)$ from the covers of $[a,b]$. As $\hat{f_k}\xrightarrow{WP_{\alpha,\beta} ^{\theta}} \hat{f}$, consider the set $A_k=\Big\{p\leq T_{\alpha\beta (k)}: t_kd(\hat{f_k}(x_i),\hat{f}(x_i))< \frac{\varepsilon}{3M}\Big\}$ for $i=1,2,\ldots, m$. Then for any $n\in A_k$ and $x\in (x_i-\delta,x_i+\delta)$ for $i=1,2,\ldots, m$, we have
\begin{eqnarray*}
t_nd(\hat{f_n}(x),\hat{f}(x)) &\leq& t_nd(\hat{f_n}(x),\hat{f_n}(x_i))+t_nd(\hat{f_n}(x_i),\hat{f}(x_i))+t_nd(\hat{f}(x_i),\hat{f}(x))\\
&<& M\Big[\frac{\varepsilon}{3M}+\frac{\varepsilon}{3M}+\frac{\varepsilon}{3M}\Big]=\varepsilon
\end{eqnarray*}
Thus we get
\begin{eqnarray*}
n\in A_k \implies t_nd(\hat{f_n}(x),\hat{f}(x))<\varepsilon\ \mbox{where $n$ and the intervals $(x_i-\delta,x_i+\delta)$ are choosen arbitrarily.}
\end{eqnarray*}
\begin{eqnarray*}
\mbox{Then} &&\ \Big\{k\leq T_{\alpha\beta (n)}: t_kd(\hat{f_k}(x),\hat{f}(x))\geq \varepsilon \ \forall \ x\in[a,b]\Big\}\\ && \subseteq \Big\{k\leq T_{\alpha\beta (n)}: t_kd(\hat{f_k}(x),\hat{f}(x))\geq \frac{\varepsilon}{3M}\ \mbox{for every}\ \ x\in[a,b]\Big\}
\end{eqnarray*}
Taking cardinality in both sides and diving by $T_{\alpha\beta (n)}$, we get $\hat{f_k}\xrightarrow{WU_{\alpha,\beta} ^{\theta}} \hat{f}$ on $[a,b]$. This completes the proof.
\end{proof}

\end{document}